\documentclass{amsart}

\usepackage{graphicx}
\usepackage{color}

\newtheorem{theorem}{Theorem}[section]
\newtheorem{lemma}[theorem]{Lemma}

\theoremstyle{definition}

\theoremstyle{remark}

\numberwithin{equation}{section}

\begin{document}

\title{ An identity concerning the Riemann-zeta function}

\author{Douglas Azevedo}
\address{Federal Technological University of Parana}
\curraddr{ Departamento Acadêmico de Matemática, Universidade Tecnológica Federal do Paraná,  Caixa Postal 238,   Av. Alberto Carazzai, 1640, Cornélio Procópio, Paraná, Brazil
.
}
\email{douglasa@utfpr.edu.br}






\begin{abstract} 
For a certain function  $J(s)$ we prove  that the identity
$$\frac{\zeta(2s)}{\zeta(s)}-\left(s-\frac{1}{2}\right)J(s)=\frac{\zeta(2s+1)}{\zeta(s+1/2)},
$$
holds in the half-plane Re$(s)>1/2$ and both sides of the equality are analytic in this half-plane.

\end{abstract}

\maketitle
\section{Introduction }

Let $s=\sigma+it$,  $\zeta(s)$, as usual,  the Riemann zeta-function and
$\lambda(n)$ the Liouville function, that is, a completely multiplicative function
with $\lambda(1)=1$ and  $\lambda(n)=-1$  at all primes $p$. 

It is well know that  $\lambda(n)$ is deeply related to the Riemann hypothesis, or simply, RH and also to  the Prime Number Theorem (PNT). For instance, RH  is equivalent 
to 
$$\sum_{n\leq x} \lambda(n)=O_{\epsilon}(x^{\frac{1}{2}+\epsilon}),$$
for all $\epsilon>0$. Whereas the PNT is 
equivalent to (\cite{Borwein} )
$$\sum_{n=1}^{\infty}\frac{\lambda(n)}{n}=0.$$
We  also refer to \cite[p.179]{Mont} for more details on the PNT.

Any improvements  in the zero-free region for  $\zeta(s)$ will immediately imply
improvements in the error term of the prime number theorem. For example,  if the Riemann
Hypothesis is true then we obtain $$\pi(x)=\int_{2}^{x}\frac{du}{\log(u)}+O(\sqrt{x}\log(x)),$$
and  this last is not only implied by the RH  but actually implies the RH itself.

The   identity
\begin{equation}\label{lambda}
\frac{\zeta(2s)}{\zeta(s)}=\sum_{n=1}^{\infty}\frac{\lambda(n)}{n^{s}},\,\,\,\textrm{Re}(s)>1,
\end{equation}
was investigated by P\'{o}lya \cite{Polya}  and  Tur\'{a}n   \cite{Turan} aiming information about the RH.
P\'{o}lya studied the change of sign of  sum $$P(x)=\sum_{n\leq x} \lambda(n) $$ and its relation to the RH.  He 
remarked that the RH would follow if one could establish that $P(x)$ eventually has constant sign. 
The assertion that $P(x)\leq 0$ for $x \geq  2$ is often called ‘P\'olya's conjecture’ in the literature, although it appears that Pólya never in fact stated this as a conjecture (\cite{MT2012}). 
Similarly, in connection with some studies of partial sums of the Riemann zeta function, 
 Tur\'{a}n  investigated 
 $$T(x)=\sum_{n\leq x} \frac{\lambda(n)}{n}. $$
He
 showed that if there exists a positive constant $c$ such that $T (n)> -c/\sqrt{n} $
 for all sufficiently large $n$, then the RH would follow.  
 He also reported that several assistants had verified that $T (n) > 0$ for $n \leq  1000$. (\cite{MT2012})

However, the main issue in both approaches is the determination of the sign of $P(x)$
and $T(x)$.  
In 1958, Haselgrove \cite{Haselgrove} proved that both $P(x)$ and $T(x)$
change sign infinitely often.  Consequently 
one can not rely on this to infer about the RH.

 These  approaches to the RH  depend upon an application of a Landau's result about the domain of convergence of Dirichlet integrals ( Lemma \ref{Landau} below).
Let us summarize the idea behind the approach  which consists of  two main features: the first,  is
to represent the ratio $\frac{\zeta(2s)}{\zeta(s)}$ by means of   a Dirichlet integral, that is, 
$$\int_{1}^{\infty}A(u)u^{-s}\,du$$
for all Re$(s)>1$. The second, is that  $A(u)$ does not change sign  for all $u$ large.
By partial summation 
it easy to see that
$$\frac{\zeta(2s)}{\zeta(s)}=s\int_{1}^{\infty}\frac{P(u)}{u}u^{-s}\,du\,\,\,\,\,\textrm{and}\,\,\,\,\frac{\zeta(2s)}{\zeta(s)}=(s-1)\int_{1}^{\infty}T(u)u^{-s}\,du,$$
Re$(s)>1$. 
Therefore, if  $P(x)$ or  $T(x)$ do not change  sign for all  $x$ large (which is not the case), by   Landau's result, 
 since $\frac{\zeta(2s)}{\zeta(s)}$ has  real singularity at
$s=\frac{1}{2}$,  then $\frac{\zeta(2s)}{\zeta(s)}$ is analytic in the half-plane Re$(s)>\frac{1}{2}$.
Hence,   $\zeta(s)\neq 0$ in this half-plane. In other words, we would have the validity of the RH. 

In this paper we follow a  similar idea to investigate the RH. Precisely, we prove that there exist
 certain functions $F(x)$ and  $J(s)$, with  $\lim_{x\to\infty} F(x)=-1$ as a consequence of the PNT, 
   for which
  the identity
$$\left(s-\frac{1}{2}\right)^{-1/2}\left(\frac{\zeta(2s)}{\zeta(s)}-1\right)+ J(s)=\int_{1}^{\infty}F(u)u^{-s-1/2}\,du,$$
holds in the half plane Re$(s)>1$. 
Hence, since,  $F(x)<0$ for all $x$ large and the integral on the 
right hand side converges absolutely for $s>1/2$ but diverges at $s=1/2$, 
as a consequence of  
 Landau's theorem for Dirichlet's integrals (Lemma \ref{Landau}) we can conclude that both sides of the identity
 are analytic 
 in the half-plane Re$(s)>1/2$. 
 
 The same ideas  also hold if we use the M\"oebius function $\mu(n)$ instead $\lambda(n)$, with some  adaptations. 

 \section{Auxiliary results}
 
 Let us present some results that will be needed  throughout this paper. 
%
%
%
%
%
%

For the next result, which is an analogue of   Landau's theorem concerning Dirichlet series with  non-negative coefficients,   
we refer  \cite[Lemma 15.1]{Mont}. This result is the main tool used to obtain the central results of this paper.

\begin{lemma}\label{Landau}
Suppose that  $G(u)$  is bounded  Riemann-integrable function  on 
every compact interval $[1,a],$ and that   $G(u)\geq 0$ for all $u> M$ or  $G(u)\leq 0$  for all $u> M$. 
Let $\sigma_c$ denote the infimum of those $\sigma$ for which $\int_{M}^{\infty}G(u)u^{-\sigma}\,du$ converges.
Then the function
$$\varphi(s)=\int_{1}^{\infty}G(u)u^{-s}\,du$$
is analytic in the half-plane Re$(s)>\sigma_{c	}$ but not  at $s=\sigma_{c}$.
\end{lemma}

Our first result is just an observation derived from the identity \eqref{lambda} and an application
of the PNT. 

\begin{lemma}\label{an}
Let  $\lambda(n)$ be the Liouville function and $a(n)$ be the arithmetic function defined as 
$$a(n)=\begin{cases}0,\,\,n=1 
\\
\lambda(n),\,\,\,n\geq 2.
\end{cases}$$
We have that,
$$\frac{\zeta(2s)}{\zeta(s)}-1=\sum_{n=1}^{\infty}\frac{a(n)}{n^{s}},\,\,\,\textrm{Re}(s)>1.
$$
In particular,
$$\sum_{n=1}^{\infty}\frac{a(n)}{n}=-1 .
$$
\end{lemma}
\begin{proof}
Since $\lambda(1)=1$, it is immediate that
$$\frac{\zeta(2s)}{\zeta(s)}=\sum_{n=1}^{\infty}\frac{\lambda(n)}{n^{s}}=1+\sum_{n=2}^{\infty}\frac{\lambda(n)}{n^{s}}=1+\sum_{n=1}^{\infty}\frac{a(n)}{n^{s}},\,\,\,\textrm{Re}(s)>1.
$$
The conclusion of the proof follows the definition of $a(n)$ and  from the fact that  the PNT is equivalent to 
$$\sum_{n=1}^{\infty}\frac{\lambda(n)}{n}=0.$$
\end{proof}

Let  $x\geq 1$ and $b(n)$ an arithmetical function. Consider the Dirichlet polynomials
\begin{equation}\label{F}
F_{x}(\alpha)=\sum_{n\leq x}b(n)n^{-\alpha},\,\,\, \alpha\in\mathbb{R}.
\end{equation}

A first consequence that can be extracted from  Lemma \ref{an} is an alternative proof for the well-known fact that $\zeta(s)\neq 0$ for 
 Re$(s)>1$.
   Indeed,  by partial summation 
 \begin{equation}\label{>1}
 \left(-1+\frac{1}{\zeta(s)}\right)(s-1)^{-1}=\int_{1}^{\infty}F_{u}(1)u^{-s}\,du,\,\,\textrm{Re}(s)>1.
  \end{equation}
Since $\lim_{x\to\infty} F_{u}(1)=-1$, it is clear that $F_{u}(1)<0$ for all $x$ large.
Also note that, since  the inequalities
\begin{equation}\label{zetareal}
\frac{1}{\sigma-1}<\zeta(\sigma)<\frac{\sigma}{\sigma-1}
\end{equation}
hold for all $\sigma>0$ \cite[p.25]{Mont}, 
the left hand side of  \eqref{>1} has a pole at $\sigma=1$ but is analytic  on the real line for $\sigma>1$. 
Hence, by an application of Lemma \ref{Landau},   \eqref{>1} 
holds for  Re$(s)>1$ and both  sides of this identity  are analytic in  this half-plane.  Therefore, $\zeta(s)\neq 0$ for Re$(s)>1$. 

Now we proceed by  presenting  an application of the Mean Value Theorem  which  plays 
an important role in  the justification of the main results that will be presented in this paper.

  \begin{lemma}\label{MVT}
  Let $\beta> \alpha$ and $b(n)$ an arithmetic function. 
There exists a sequence $\xi=(\xi(n))\subset]\alpha,\beta[$ such that
$$F_{x}(\beta)-F_{x}(\alpha)=-(\beta-\alpha)\sum_{n\leq x}\frac{b(n)\log(n)}{n^{\xi(n)}},$$  
  for all $x\geq 1$. Moreover, $\xi$ decreases and  $$\lim_{n\to\infty}\xi(n)= \alpha.$$

  \end{lemma}
  \begin{proof} Let $n>  1$ and $\beta> \alpha$. By the Mean Value Theorem
  $$n^{-\beta}-n^{-\alpha}=-(\beta-\alpha)\log(n)n^{-\xi(n)},$$ 
  for some $\xi(n)\in]\alpha,\beta[$. That is, for each $n\geq 1$,  there exists $\xi(n)\in]\alpha,\beta[$ for which 
    $$F_{x}(\beta)-F_{x}(\alpha)=-(\beta-\alpha)\sum_{n\leq x}\frac{b(n)\log(n)}{n^{\xi(n)}},$$  
for all $x\geq 1$.
  
  The limit and the decreasigness  of $\xi$ follows from  the equality
$$\xi(n)=\frac{\log\left( (\beta-\alpha)\log(n)\frac{n^{\alpha+\beta}}{n^{\beta}-n^{\alpha}}\right)}{\log(n)}=\frac{\log(\log(n))}{\log(n)}+ \frac{\log\left( \frac{(\beta-\alpha) }{1-n^{\alpha-\beta}}\right)}{\log(n)}+\alpha,$$
for $n>1$.

  \end{proof}

%

\section{Main results}

Let $a(n)$ be as in Lemma \ref{an} and we now fix the notation $$F_{x}(\alpha)=\sum_{n\leq x}\frac{a(n)}{n^\alpha},$$
with $\alpha>0 $ and $x\geq 1$.

\begin{lemma}\label{integral} Let Re$(s)>1$.
The following identity holds true
$$\left(s-\frac{1}{2}\right)^{-1}\left(\frac{\zeta(2s)}{\zeta(s)}-1\right)=\int_{1}^{\infty}F_{u}(1/2) u^{-s-1/2}\,du.$$
\end{lemma}
\begin{proof}

Let $x\geq 1$. By partial summation 
we obtain that
$$\sum_{n\leq x}\frac{a(n)}{n^{s}}=F_{x}(1/2)x^{-s+1/2}+\left(s-\frac{1}{2}\right)\int_{1}^{x}F_{u}(1/2) u^{-s-1/2}\,du,
$$
for Re$(s)>1$. Since 
$$\lim_{x\to\infty}F_{x}(1/2)x^{-s+1/2}=0,$$
whenever Re$(s)>1$, we obtain that 
$$\sum_{n=1}^{\infty}\frac{a(n)}{n^{s}}=\left(s-\frac{1}{2}\right)\int_{1}^{\infty}F_{u}(1/2)u^{-s-1/2}\,du,
$$
holds in the half-plane Re$(s)>1$.
By  Lemma \ref{an} we have that
$$\frac{\zeta(2s)}{\zeta(s)}-1=\left(s-\frac{1}{2}\right)\int_{1}^{\infty}F_{u}(1/2)u^{-s-\frac{1}{2}}\,du,
$$
holds in the half-plane Re$(s)>1$. 

\end{proof}

In the following  result  we prove an identity that concerning the analiticity of $\frac{\zeta(2s)}{\zeta(s)}$ in the half-plane Re$(s)>1/2$.

\begin{theorem}\label{Main}  
There exists  a decreasing  sequence $\xi=(\xi(n))\subset ]1/2,1[$ such that 
$$\lim_{n\to\infty}\xi(n)= 1/2$$ 
for which
\begin{equation}\label{eqq2}
\left(s-\frac{1}{2}\right)^{-1}\left(\frac{\zeta(2s)}{\zeta(s)}-1\right)-J_{\xi}(s)=\int_{1}^{\infty}F_{u}(1)u^{-s-1/2}\,du,
\end{equation}
holds for  Re$(s)>1/2$ and both sides  of the equality are analytic in this  half-plane. In the equality above we have that 
$$J_{\xi}(s)=\int_{1}^{\infty}L_{u}(\xi)u^{-s-1/2}\,du,\,\,\,\,\textrm{Re}(s)>1$$
and $$L_{x}(\xi)=\frac{1}{2}\sum_{n\leq x} \frac{a(n)\log(n)}{n^{\xi(n)}},\,\,\,x\geq 1.$$
\end{theorem}
\begin{proof}
From  the previous lemma  we have that
\begin{equation}\label{eqq1}
\left(s-\frac{1}{2}\right)^{-1}\left(\frac{\zeta(2s)}{\zeta(s)}-1\right)=\int_{1}^{\infty}F_{u}(1/2)u^{-s-1/2}\,du,
\end{equation}
in the half-plane Re$(s)>1$. 
For  $\beta =1$ and $\alpha=1/2$ in  Lemma \ref{MVT}, there exists a decreasing sequence   $\xi=(\xi(n))\subset ]1/2, 1[$ 
such that 
$$\lim_{n\to\infty}\xi(n)= 1/2$$ 
for which
$$F_{u}(1/2)=F_{u}(1)+L_{u}(\xi),$$
for all $u\geq 1$, 
where $$L_{u}(\xi)=\frac{1}{2}\sum_{n\leq u} \frac{a(n)\log(n)}{n^{\xi(n)}},\,\,\,u\geq 1.$$
 Hence,  we can rewrite \eqref{eqq1} as 
\begin{equation}\label{eqq4}
\left(s-\frac{1}{2}\right)^{-1}\left(\frac{\zeta(2s)}{\zeta(s)}-1\right)-J_{\xi}(s)=\int_{1}^{\infty}F_{u}(1)u^{-s-1/2}\,du,
\end{equation}
for  Re$(s)>1$, for some sequence  $\xi=(\xi(n))\subset]1/2, 1[$,  with
$$J_{\xi}(s)=\int_{1}^{\infty}L_{u}(\xi)u^{-s-1/2}\,du.$$

In order to show that the equality	 \eqref{eqq4} extends to the half-plane Re$(s)>1/2$  and that both sides are analytic there, first note that 
 from  Lemma \ref{an} ,  $\lim_{u\to\infty}F_{u}(1)=-1$. 
Thus  there exists $M\geq 1$ for which
$F_{u}(1)<0$ for all $u>M$.  Moreover,  clearly the integral
$$\int_{1}^{\infty}F_{u}(1)u^{-\sigma-1/2}\,du$$
converges (absolutely) at every $\sigma>1/2$ but diverges at $\sigma=1/2$. 
By Lemma \ref{Landau}, this implies that the function
$$\int_{1}^{\infty}F_{u}(1)u^{-s-1/2}\,du$$
is analytic for Re$(s)>1/2$.  Therefore, 
 Lemma \ref{Landau} implies that \eqref{eqq4} 
holds for Re$(s)>1/2$ and both sides of \eqref{eqq2} are analytic in this  half-plane.
\end{proof}

Note that by Lemma \ref{Landau}, if $F_{x}(1/2)\leq 0$ for all $x$ large, then by   Lemma \ref{integral}
$\frac{\zeta(2s)}{\zeta(s)}$ is analytic in the half-plane Re$(s)>1/2$, which implies the truth of  RH. 
In the following result it is provided another condition to obtain the RH.

\begin{theorem} If there exists $r>0$ such that
$L_{x}(\xi)\leq 1-r$ for all $x$ large, then 
$\frac{\zeta(2s)}{\zeta(s)}$ is analytic in Re$(s)>1/2$. 
\end{theorem}
\begin{proof}
From Lemma \ref{MVT},  we have that
$F_{x}(1/2)=F_{x}(1)+L_{x}(\xi)$, for all $x\geq 1$. Lemma \ref{an} implies that
 for any $\epsilon>0$,  $F_{x}(1)<-1+\epsilon$, for all $x$ large. Hence, if $L_{x}(\xi)\leq 1-r$,
 for $0<\epsilon\leq r$ sufficiently small   
 $F_{x}(1/2)=F_{x}(1)+L_{x}(\xi)<\epsilon-r\leq 0$, for all $x\geq 1$. By Lemma \ref{Landau} and Lemma \ref{integral} this implies
 the analyticity of $\frac{\zeta(2s)}{\zeta(s)}$ in the half-plane Re$(s)>1/2$.
\end{proof}

Now note that  by partial summation, it follows that
$$\sum_{n=1}^{\infty}\frac{a(n)}{n^{s+1/2}}=(s-1/2)\int_{1}^{\infty}F_{u}(1)u^{-s-1/2}\,du$$
for Re$(s)>1/2$. Hence,
 as a consequence of Lemma \ref{Main} and Lemma \ref{an}, equation
\eqref{eqq2} can be writen as
\begin{equation}\label{eqq2XX}
\frac{\zeta(2s)}{\zeta(s)}-\left(s-\frac{1}{2}\right)J_{\xi}(s)=\sum_{n=1}^{\infty}\frac{\lambda(n)}{n^{s+1/2}},
\end{equation}
Re$(s)>1/2$, and both sides are analytic in this half-plane. 
Moreover, from \eqref{lambda},  this previous  conclusions imply that

\begin{theorem}\label{zetaIdentidade}
\begin{equation}\label{eqq2XX}
\frac{\zeta(2s)}{\zeta(s)}-\left(s-\frac{1}{2}\right)J_{\xi}(s)=\frac{\zeta(2s+1)}{\zeta(s+1/2)},
\end{equation}
Re$(s)>1/2$ and both sides are analytic in this half-plane. 
\end{theorem}

\bibliographystyle{amsplain}

\end{document}